\documentclass[twoside]{article}
\usepackage[a4paper]{geometry}
\usepackage[latin1]{inputenc} 
\usepackage[T1]{fontenc} 
\usepackage{RR}
\usepackage{graphicx} 
\usepackage{marginnote}
\usepackage{amsmath,amssymb}
\usepackage{ulem}

\newtheorem{Property}{Property}

\newcommand{\beq}{\begin{equation}}
\newcommand{\eeq}{\end{equation}}

\newtheorem{thm}{Theorem}
\newtheorem{re}{{\bf Remark}}
\newtheorem{ex}{{\bf Example}}
\newtheorem{cor}{{\bf Corollary}}
\newtheorem{defi}{{\bf Definition}}
\newtheorem{prop}{{\bf Proposition}}
\newtheorem{lemme}{{\bf Lemma}}
\newtheorem{prob}{{\bf Problem}}
\newtheorem{ass}{{\bf Assumption}}
\newtheorem{notat}{{\bf Notation}}
\newtheorem{solut}{{\bf Solution}}
\newenvironment{theorem}{\begin{thm} \rm}{\end{thm}}

\newenvironment{proposition}{\begin{prop} \rm}{\end{prop}}

\newenvironment{assumption}{\begin{ass} \rm}{\end{ass}}

\newenvironment{solution}{\begin{solut} \rm}{\end{solut}}
\newenvironment{proof}{\paragraph{Proof:}}{\endpf\\}
\def\endpf{\hfill$\Box$\medskip}

\def \RR {I \! \! R}

\def \ee {\begin{equation}}
\def \eee {\end{equation}}
\def \eqe {\begin{eqnarray}}
\def \eqee {\end{eqnarray}}

\RRNo{7929}
\RRdate{April 12th, 2012}
\RRauthor{
Fr\'ed\'eric Grognard%
  \thanks[BIOCORE]{BIOCORE, INRIA Sophia
   Antipolis, BP 93, 06902 Sophia Antipolis Cedex, France. {\tt\small
           \{frederic.grognard,\,olivier.bernard\}@sophia.inria.fr} }%
  \and
Andrei R. Akhmetzhanov\thanks{Department of Biology, McMaster University,
1280 Main St W, Hamilton ON, L8S 4K1, Canada. {\tt\small
           akhmetzhanov@googlemail.com}}%
\and Olivier Bernard\thanksref{BIOCORE} }
\authorhead{Grognard, Akhmetzhanov \& Bernard}
\RRtitle{Commande p\'eriodique optimale pour la maximisation de la productivit\'e de biomasse dans un photobior\'eacteur utilisant la lumi\`ere naturelle}
\RRetitle{Periodic optimal control for biomass productivity maximization in a photobioreactor using natural light}
\titlehead{Biomass productivity maximization in a photobioreactor }
\RRresume{Nous abordons la question de l'optimisation de la productivit\'e de biomasse microalgale 
dans le cadre de la production en photobior\'eacteurs sous
l'influence du cycle jour/nuit. Pour cela, nous proposons un mod\`ele simple de bior\'eacteur qui tient compte de l'att\'enuation lumineuse dans le r\'eacteur li\'ee \`a la densit\'e \'elev\'ee de biomasse
et une loi de commande qui optimise la productivit\'e sur une seule journ\'ee par l'application du principe du maximum de Pontryagin, avec la dilution comme commande principale. Une contrainte importante sur la solution obtenue est
que la biomasse dans le r\'eacteur doit \^etre au m\^eme niveau 
au d\'ebut et \`a la fin du cycle de telle sorte que le m\^eme contr\^ole peut \^etre appliqu\'e
tous les jours et optimise une certaine forme de la productivit\'e \`a long terme. Plusieurs sc\'enarios
optimaux sont possibles en fonction des param\`etres de croissance de la micro-algues et de la valeur maximale admissible du taux de dilution: bang-bang, bang-singuli\`ere-bang ou, si le taux de croissance des algues est tr\`es fort en pr\'esence de lumi\`ere,
dilution maximale constante. Un diagramme de bifurcation est pr\'esent\'e pour illustrer
pour quelles valeurs de param\`etres ces comportements diff\'erents se produisent. Enfin, une
strat\'egie sous-optimale bang-bang simple est propos\'ee, qui permet d'approcher la productivit\'e de la strat\'egie optimale.
}
\RRabstract{
We address the question of optimization of the microalgal biomass long term productivity
in the framework of production in photobioreactors under
the influence of day/night cycles. For that, we propose a simple bioreactor
model accounting for light attenuation in the reactor due to biomass density
and obtain the control law that optimizes productivity over a single day
through the application of Pontryagin's maximum principle, with the dilution
rate being the main control. An important constraint on the obtained solution
is that the biomass in the reactor should be at the same level at the
beginning and at the end of the day so that the same control can be applied
everyday and optimizes some form of long term productivity. Several scenarios
are possible depending on the microalgae's strain parameters and the maximal
admissible value of the dilution rate: bang-bang or bang-singular-bang control
or, if the growth rate of the algae is very strong in the presence of light,
constant maximal dilution. A bifurcation diagram is presented to illustrate
for which values of the parameters these different behaviors occur. Finally, a
simple sub-optimal bang-bang strategy is proposed that numerically achieves productivity
levels that almost match those of the optimal strategy.   }
\RRmotcle{Commande optimale, commande p\'eriodique, photobior\'eacteur, productivit\'e, biofuel}
\RRkeyword{Optimal control, periodic control, photbioreactor, productivity, biofuel}
\RRprojets{Biocore}
\RCSophia 

\begin{document}
\makeRR   
\section{Introduction}

Microalgae  have recently received more and more attention in the frameworks
of CO$_2$ fixation and renewable energy \cite{Huntley2007, Chisti2007}. Their
high actual photosynthetic yield compared to terrestrial plants (whose growth
is limited by CO$_2$ availability) leads to large potential algal biomass
productions  of several tens of tons per hectare and per
year \cite{Chisti2007}. Also, they have a potential for numerous high added value
commercial applications \cite{Spolaore200687} as well as wastewater treatment
capabilities including nutrients and heavy metal removal
\cite{Hoffmann1998,Oswald1988}. 
We focus on the industrial production of microalgae in the so called photobioreactors. In this work, photobioreactor is meant in the large, including  many possibles types of culturing devices, including the most simple open raceways systems.

Concentrating on the bioenergy applications of
microalgae, one central feature is the energy balance of the process; the only
conceivable light source ({\it i.e} the primary energy source) is natural
sunlight, which varies during the day. This variation might have an important
impact on the microalgae productivity; it is therefore necessary to take it
into account in order to design control laws that will be applied to the
bioreactor.  

The objective of this paper is to develop an optimal control law that would
maximize the yield of a photobioreactor operating in continuous mode, while
taking into account that the light source  that will be used is the natural
light. The light source is therefore periodic with a light phase (day) and a
dark phase (night). In addition to this time-varying periodic light source, we
will take the auto-shading in the photobioreactor into account: the pigment
concentration (mainly chlorophyll) affects the light distribution and thus the
biological activity within the reactor. As a consequence, for too high a
biomass, light in the photobioreactor is strongly attenuated and per-capita
growth is low. 

Optimal control of bioreactors has been studied for many years whether it was
for metabolites production \cite{Tartakovsky1995}, ethanol fermentation
\cite{Wang1997}, baker yeast production \cite{Wu1985} or, more generally,
optimal control of fed-batch processes taking kinetics uncertainties into
account \cite{Smets2004}. The control of photobioreactors is however a lot
more scarce in the literature, though the influence of self-shading on the
optimal setpoint \cite{Masci2010} or on an MPC control algorithm
\cite{Becerra2008} for productivity optimization has already been studied. The
light-variation was mostly absent \cite{Masci2010,Becerra2008} or considered
to be an input that could be manipulated in order to impose the physiological
state of the microalgae \cite{Marxen2005} or maximize productivity as one of
the parameters of bioreactor design \cite{Suh2003}. The present problem has
however not been tackled yet in the literature. 

We therefore developed a model that takes both the light variation and the
self-shading features into account in order to develop the control law, where
the substrate concentration in the input (marginally) and the dilution rate
(mainly) will be used. This model should not be too complicated in order to be
tractable and should present the main features of the process. Since we want
to develop a control strategy that will be used on the long run, we could
choose an infinite time-horizon measure of the yield. However, we rather take
advantage of the observation that, in the absence of a discount rate in the
cost functional, the control should be identical everyday and force the state
of the system to be identical at the beginning of the day and 24 hours
later. We therefore opted for optimizing a cost over one day with the
constraint that the initial and terminal states should be identical. 

The paper is structured as follows: in Section~\ref{sec:model} we develop a
photobioreactor model presenting all the aforementioned features; in Section
\ref{sec:constant} we identify the optimal operating mode in constant light
environment; in Section \ref{sec:varying} we develop our main result, that is
the form of the optimal control law in a day \& night setting; in Section
\ref{sec:large} we identify the consequence if the control constraint is
generous; we conclude by a simulation study and bifurcation analysis in
Section \ref{sec:simulations}. 

\section{A photobioreactor model with light attenuation}\label{sec:model}

Micro-algae growth in a photobioreactor is often modeled through one of two
models, the Monod model \cite{Monod1942} or the Droop Model
\cite{Droop1968}. The latter is more accurate as it separates the process of
substrate uptake and growth of the microalgae. However, the former already
gives a reasonable representation of reality by coupling growth and uptake,
and is more convenient for building control laws since it is simpler. The
Monod model writes: 
\begin{equation}\label{monod}
\left\{
\begin{array}{lll}
  \dot S& =& D(S_{in}-S)-k\mu(S,I)X\;,\\
  \dot X& =& \mu(S,I)X-DX\;,
\end{array}
\right.
\end{equation}
where $S$ and $X$ are the substrate (e.g. nitrate concentration) and biomass  (measured in carbon concentration, gC.L$^{-1}$) in the medium,
which is supposed to be homogeneous through constant stirring, while $D$ is the dilution rate, $S_{in}$ is the substrate input concentration and $k$ is the substrate/biomass yield coefficient; $\mu(S,I)$ is the microalgae biomass growth rate, depending on substrate and light intensity $I$, which is often taken to be of the Michaelis-Menten form with respect to $S$ and $I$: 
\begin{equation}\label{michaelis-menten}
\mu(S)=\frac{\bar\mu S}{S+K_S}\frac{ I}{I+K_I} \;.
\end{equation}
With $K_S$ and $K_I$ the associated half saturation constants.
We will however not focus on this specific form and simply consider that $\mu(S,I)=\mu_S(S)\mu_I(I)$ is the product of a light related function $\mu_I(I)$ and of a substrate related function $\mu_S(S)$, such that $\mu_S(0)=0$, is increasing and bounded  (with $\lim_{S\rightarrow\infty}\mu_S(S)=\bar \mu $). Generalization for the function $\mu_I(I)$ will also be proposed later on. 

We will now provide extensions to this model so that it fits better to the production problem in high-density photobioreactors under study. 

\subsection{Adding respiration}

The central role played by respiration in the computation and optimization of
the productivity of a photobioreactor has been known for a long time (see
e.g. \cite{GroSoe85}), since it tends to reduce the biomass. We therefore
introduce respiration by the microalgae into our model. In contrast to
photosynthesis, this phenomenon takes place with or without light: from a
carbon point of view, it converts biomass into carbon dioxyde, so that we
represent it as a $-r X$ term that represents loss of biomass. The biomass
dynamics then become 
\[
 \dot{X} = \mu(S) X -rX -D X\;. 
\]
Remark that mortality is also included in the respiration term. Mortality can be high in photobioreactors, due to its high biomass density.

\subsection{Adding light attenuation}

When studying high concentration reactors, light
cannot be considered to have the same intensity in any point of the reactor. We consider an horizontal planar
photobioreactor (or raceway) with constant horizontal section $A$ over the
height of the reactor and vertical incoming light. We will then represent
light attenuation following an exponential Beer-Lambert law \cite{Lambert1760}
where the attenuation at some depth $z$ comes from the total biomass $Xz$ per
surface unit contained in the layer of depth $[0,z]$: 
\begin{equation} \label{eq:Iz}
I(Xz) = I_0\mathrm e^{-aXz}\;,
\end{equation}
where $I_0$ is the incident light and $a$ is a light attenuation
coefficient. In microalgae, chlorophyll is mostly the cause of this shadow
effect and, in model (\ref{monod}), it is best represented by a fixed portion
of the biomass \cite{Bernard2009}, which yields the direct dependence in $X$
in model (\ref{eq:Iz}).
With such an hypothesis on the light intensity that reaches depth $z$, growth
rates vary with depth: in the upper part of the reactor, higher light causes
higher growth than in the bottom part. Supposing that light attenuation
directly affects the growth rate \cite{Huismann2002}, the growth rate
in the form (\ref{michaelis-menten}) for a given depth $z$ can then be written
as 
\[
\displaystyle \mu_z(S,I(Xz)) = \frac{ I(Xz)}{I(Xz)+K_I}  \mu_S(S)\;.
\]
We see that this results in a differentiated biomass growth-rate in the
reactor which could possibly yield spatial heterogeneity in the biomass
concentration; however, constant stirring is supposed to keep  the
concentrations of $S$ and $X$ homogeneous.  Then, we can compute the mean growth rate in the reactor:
\[
\mu(S,I_0,X) = \frac{1}{L} \int_0^L \mu_z(S,I(Xz))\mathrm dz\;,
\]
where $L$ is the depth of the reactor. It is this average growth rate that will be used
in the lumped model that we develop.   
We then have:
$$
\begin{array}{lll}
\mu(S,I_0,X)&=& \displaystyle\frac{1}{L}\int_0^L\frac{I_0\mathrm e^{-aXz}}{I_0\mathrm e^{-a Xz}+K_I}\>\mathrm dz  \mu_S(S) \\
&=& \displaystyle\frac{1}{aXL}\ln\left(\!\frac{I_0+K_I}{I_0\mathrm e^{-aXL}+K_I}\!\right)\!\mu_S(S)\;.
\end{array}
$$

\subsection{Considering varying light}

In order to determine more precisely the model, we should now indicate what the varying light is like. Classically, it is considered that daylight 
varies as the square of a sinusoidal function so that
\[
I_0(t)=\left(\max\left\{\sin\left(\frac{2\pi t}{{T}}\right),0\right\}\right)^2\;,
\]
where $T$ is the length of the day. The introduction of such a varying light would however render the computations analytically untractable. Therefore, we approximate the light source by a step function:
$$
I_0(t)=\left\{\begin{array}{lllll}
  \bar I_0,&\  & 0\le t < \bar {T}&\mbox{\quad ---\quad light  phase}\;,\\
  0,&\  & \bar {T}\le t < {T}&\mbox{\quad ---\quad dark phase}\;.
\end{array}\right.
$$
In a model where the time-unit is the day, ${T}$ will be equal $1$. At the equinoxes, we have that $\bar {T}=\frac{{T}}{2}$, but this quantity obviously depends on the time of the year.

\subsection{Reduction and generalization of the model}

The system for which we want to build an optimal controller is therefore 
\begin{equation}
\label{our_monod}
\left\{
\begin{array}{lll}
  \dot S& =&\displaystyle D(S_{in}-S)-k \frac{1}{aXL}\ln\left(\!\frac{I_0(t)+K_I}{I_0(t)\mathrm e^{-aXL}+K_I}\!\right)\! \mu_S(S) X\;,\\
  \dot X& =& \displaystyle\frac{1}{aXL}\ln\left(\!\frac{I_0(t)+K_I}{I_0(t)\mathrm e^{-aXL}+K_I}\!\right)\!\mu_S(S)X-rX-DX\;.
\end{array}
\right.
\end{equation}
However, in order to maximize the productivity it is clear that the larger $S$
the better: since it translates into larger growth rates of the biomass
$X$. Hence, the control of the inflow concentration of the substrate $S_{in}$
should always be kept very large so as to always keep the substrate in the
region where $\mu_S(S)\approx \bar \mu$. We can then concentrate on the reduced
model 
\[
  \dot X = \frac{\bar \mu}{a L}\ln\frac{I_0(t)+K_I}{I_0(t)\mathrm e^{-aXL}+K_I}-rX-DX\;,
\]
where the only remaining control is the dilution $D$ and which then
encompasses all the relevant dynamics for the control problem. 
As was seen in \cite{Masci2010}, the relevant concentration in the photobioreactor, with
Beer-Lambert light attenuation, is not the volumic density $X$ but rather the surfacic
density $x=XL$: the evolution of this quantity is indeed independent of the reactor's
depth: whether we consider a deep and very diluted reactor or a shallow high
concentration reactor with identical surfacic density  
\begin{equation}\label{reduced-monod}
\begin{array}{lll}
\dot x&=& \frac{\mathrm d(XL)}{\mathrm dt}=L\left(\frac{\bar\mu}{aL}\ln\frac{I_0(t)+K_I}{I_0(t)\mathrm e^{-aXL}+K_I}-rX-DX\right)\\
&=& \displaystyle\frac{\bar\mu}{a}\ln\displaystyle\frac{I_0(t)+K_I}{I_0(t)\mathrm e^{-ax}+K_I}-rx-Dx\;,
\end{array}
\end{equation}
which is indeed independent of the depth $L$.

The reduced model (\ref{reduced-monod}) can be rewritten by taking advantage of the special form of varying light source as follows
\begin{equation}\label{reduced-monod-bang}
\dot x= \left\{\begin{array}{llll}
\displaystyle f(x)-rx-Dx,&\  & 0\le t < \bar {T}\;,&\\
   -rx-Dx,&\  & \bar{T}\le t<{T}\;.&
\end{array}\right.
\end{equation}
with 
\[f(x)=  \frac{\bar \mu}{a}\ln\frac{\bar I_0+K_I}{\bar I_0\mathrm e^{-ax}+K_I}\].

This model is quite simple except for the only nonlinear term which directly
comes from the form of $\mu_I$ and from the Beer-Lambert law (\ref{eq:Iz}). In order to generalize our approach to more light responses ({\it e.g.} for high density photobioreactor with possible high light inhibition , see \cite{Bernard2011}) and to not restrict ourselves to the Beer-Lambert law, we notice that the function
$f(x)$ is zero in zero,
increasing in $0$, bounded and strictly concave in $x$. Such a choice of $f(x)$ yields the following  somewhat trivial property, that will prove important in the following, and that is linked to $f(x)$ being concave, with $f(0)=0$ and $f'(0)>0$:
\begin{Property}\label{PropTRIVIALE}
With $f:\RR^+\rightarrow \RR$ strictly concave and such that  $f(0)=0$ and $f'(0)>0$, we have that, for any $x>0$:
\[
f'(0)>\frac{f(x)}{x}>f'(x)\;,
\]
which also shows that $\liminf_{x\rightarrow+\infty}f'(x)\leq 0$. 

Moreover \[
 \frac{\mathrm d}{\mathrm dx}\left(\frac{f(x)}{x}\right)<0\;.
\]
\end{Property}

\begin{proof} We have $f(x)=f(0)+\int_0^x f'(s)\mathrm ds$ with $f(0)=0$ and $f'(0)>f'(s)>f'(x)$ for all $0<s<x$ due to the concavity, so that $f(x)<\int_0^x f'(0)\mathrm ds=f'(0)x$ and $f(x)>\int_0^x f'(x)\mathrm ds=f'(x)x$. 

The decreasing of $\frac{f(x)}{x}$ comes from the explicit computation $\frac{\mathrm d}{\mathrm dx}\left(\frac{f(x)}{x}\right)=\frac{f'(x)x-f(x)}{x^2}$ which is negative thanks to the previous property. 
\end{proof}

The four properties that we impose on $f(x)$ should
be expected from the net-growth in a photobioreactor: no growth should take place in the absence of biomass, it should be increasing
because additional biomass should lead to more growth (at least at low densities); it should be bounded
because, when $x$ is very large the bottom of the reactor is in the dark, so
that adding new biomass simply increases the dark zone without allowing additional growth; and the per-capita growth-rate $\frac{f(x)}{x}$ should be decreasing  because additional biomass slightly degrades the environment
for all because of the shadowing it forces.

As a generalization of the model (\ref{reduced-monod}), we consider that, during the day of length $\overline{T}<T$, the system is written as 
\begin{equation}\label{day}
\dot x=f(x)-rx-ux\;,
\end{equation}
with $f(0)=0, f'(x)>0, f''(x)< 0$, $f(x)$ bounded, $r>0$ and $x,u\,\in\RR^+$ and, during the night of length $T-\overline{T}$, as
\begin{equation}\label{night}
\dot x=-rx-ux\;.
\end{equation}
In order to couple both these systems, we define $h(t)$ as a step function whose value is $1$ for $t<\overline{T}$ and $0$ for $t\geq \overline{T}$ that will allow to synthetize (\ref{day})-(\ref{night}) in the form
\begin{equation}\label{total}
\dot x=f(x)h(t)-rx-ux\;.
\end{equation}
As we will see, the central property that will be exploited is the strict concavity of $f(x)$; in the following, we will use the simpler term ``concavity'' to denote that property.

Lastly, in practice $u$ cannot take any value: it should be positive, but also upper-bounded since its value is determined by the maximum capacity of some pumps. In the following, we will consider that $0\leq u(t)\leq \bar{u}$ at all times (with $\bar u>0$).  

\section{Productivity optimization in constant light environment}\label{sec:constant}

In a constant light environment ({\it i.e.} $h(t)$ is constant), we will be able to exploit the system at a
steady state. In that case, the dynamics are strictly imposed by (\ref{day})
and the constant $u$ and corresponding equilibrium value $x^*(u)$ are chosen
in such a way that the maximum of the biomass flow rate at the output of the
reactor with the upper surface $A$ is reached. This is defined as $\max_{u}
(uA) x^*(u)$ and constitutes this productivity analysis.   

Since $ux=f(x)-rx$ at equilibrium, maximizing $uA x$ amounts to maximizing $f(x)-rx$. $f(x)$ being concave, this is equivalent to find the unique solution of 
\[
f'(x)=r\;,
\]
when it exists. Noting that $\lim_{x\rightarrow+\infty}f'(x)\leq 0$ and that
$f'(.)$ (see Property \ref{PropTRIVIALE}) is a decreasing function of $x$ because of the concavity of $f$, this
last equality has a positive solution if $f'(0)>r$ is satisfied. This
condition is crucial since, without it, the only equilibrium of the system is
$x=0$, independently of the choice of $u$. Under this assumption, the biomass
density that corresponds to the maximization of the productivity in a constant
light environment is  
\begin{equation}\label{xsigma}
x_\sigma=(f')^{-1}(r)\;.
\end{equation}
so that  that $f(x)-rx$ is increasing for $x<x_\sigma$ and decreasing for $x>x_\sigma$, which will extensively use in the following. We then obtain
\begin{equation}\label{usigma}
u_\sigma=\frac{f(x_\sigma)}{x_\sigma}-r=\frac{f\left[(f')^{-1}(r)\right]}{(f')^{-1}(r)}-r\;,
\end{equation}
which is positive by definition of $x_\sigma$. In that case, the optimal instantaneous surfacic productivity $u_\sigma x_\sigma= f(x_\sigma)-r x_\sigma$

Taking into account the actuator constraint, if $u_\sigma\leq \bar{u}$, it yields the optimal productivity. Indeed, since it satisfies
\[
 \frac{f(x^*(u))}{x^*(u)}-r=u\;.
\]
with  $\frac{f(x)}{x}$ decreasing because of Property \ref{PropTRIVIALE}, $x^*(u)$ is a decreasing function which is larger than $x_\sigma$ for $u<u_\sigma$. The productivity $ux^*(u)=f(x^*(u))-rx^*(u)$ is then an increasing function of $u$ because $f(x)-rx$ is decreasing as long as $u<u^{\sigma}$ and the optimal  productivity is obtained with $u=\bar{u}$. 

For convenience, we will define two other equilibria beyond $x_\sigma$: the equilibria of (\ref{day}) for $u=0$ and $u=\bar{u}$, that we will denote $\bar x^0$ and $\bar x^{\bar{u}}$ respectively. 


Note that the study of the present section is in line with our work in
\cite{Masci2010}, where we considered the productivity optimization in a
constant light for a Droop model with light attenuation. In that study, the
analysis was much complicated by the link between shading and nitrogen content
of the algae, so that both $S_{in}$ and $D$ had to be taken into account.


\section{Productivity optimization in day/night environment}\label{sec:varying}

In an environment with varying light,a non constant input must be considered. Here we consider the case where the photobioreactor is operated on the long term, with a daily biomass production from the reactor outlet. The problem that we thus consider is the 
maximization of the biomass production over a single day 
\[
\max_{u(t)\,\in\,[0,\bar{u}]} \int_0^T (u(t)A) x(t)\mathrm dt\;,
\]
%

We are then looking for a periodic regime, where, the photobioreactor is operated identically each day. This means that the initial condition at the beginning of the day should equal the final condition at the end of the day. This then requires
that we add the constraint 
\[
x(T)=x(0)\;.
\]
In actual applications, the length of the bright phase will change slightly
from one day to the next. This will probably impose a slight change of biomass
at the beginning of the next day but, in this preliminary study, we will
consider that such a phenomenon has little effect on the qualitative form of
the solutions.   

We therefore are faced with the following Optimal Control Problem:
\begin{equation}\label{OCP}
\begin{array}{l}
\displaystyle\max_{u(t)\,\in\,[0,\bar{u}]} \int_0^T u(t) x(t)\mathrm dt\\
\hspace*{1cm}\mbox{with}\hspace*{0.5cm}\dot x = f(x)h(t)-rx-ux\;,\\
\hspace*{2.15cm}x(T)=x(0)\;,\\
\end{array}
\end{equation}. 

\subsection{Dealing with the T-periodicity}

In order to solve this problem, it is convenient to observe that $x(T)=x(0)=x_0$
cannot be achieved for all values of $x_0$ even without requiring
optimality. For that, we consider the best case scenario, that is the one where $u=0$ for all $t$, which yields the
largest value of $x(T)$ since no biomass is taken out of the system (this is
also seen by comparing the system  with $u=0$ to systems with any $u=u(t)$); the value of
$x(T)$ obtained in the closed photobioreactor must then be larger than $x_0$
for (\ref{OCP}) to have a chance to have a solution that starts in that $x_0$.  
In this section, we will give a condition that guarantees that the set of initial conditions $x_0$ that yield $x(T)>x_0$ in a closed photobiotreactor is contained in an interval $\mathcal I=[0,x_{0max}]$, inside which the initial condition in the solution of (\ref{OCP}) will lie.  
 
The first observation that we can make is that, for $x_0\geq \bar x^0$ (the equilibrium of system (\ref{day}) with $u=0$), we have $x(T)<x_0$ because $x(t)$ is then decreasing all along the solution. The upper-bound $x_{0max}$, if it exists, is therefore smaller than $\bar x^0$. 

In order to guarantee that the interval is non-empty (or equivalently that $x_{0max}>0$) we will then concentrate on what happens for $x_0$ 
in a small neighborhood of $0$. The (\ref{day}) dynamics with
  $u=0$ can then be rewritten as 
\[
\dot x=( f'(0)-r)x\;,
\]
so that $x(\overline{T})=x_0\mathrm e^{(f'(0)-r)\overline{T}}$ and $x(T)=x(\overline{T})\mathrm e^{-r(T-\overline{T})}=  x_0\mathrm e^{( f'(0)-r)\bar T}\mathrm e^{-r(T-\overline{T})}$. We then have $x(T)>x_0$ for $x_0$ small if the exponential factor is larger than $1$, that is if: 
\begin{assumption}\label{fprimemin} The growth function $f(x)$ is $\mathcal{C}^1$, bounded, satisfies $f(0)=0$, $f'(0)>0$, $f''(x)<0$ for all $x\geq 0$ and 
\begin{equation}\label{mumin}
 f'(0) \bar{T}>rT\;.
\end{equation}
\end{assumption}
This condition is quite natural since it imposes that, when the population is
small, that is when the per capita growth rate is the largest, growth during
the daylight period exceeds the net effect of respiration that takes place
We have shown that this condition is sufficient but it is also necessary for $x_{0max}>0$ since Property \ref{PropTRIVIALE} imposes that
\[
 \dot x=f(x)h(t)-rx-ux\leq (f'(0)h(t)-r-u)x
\]
so that, for a given $x_0$, $x(T)$ in a closed photobioreactor is always smaller with the nonlinear system than the value $x^l(T)$ obtained with its linear upper approximate. A necessary condition for $x(T)>x_0$ is therefore $x^l(T)>x_0$, which amount to Assumption \ref{fprimemin}, which is therefore necessary for $x_{0max}>0$. 

Further properties are summed-up in the following proposition
\begin{prop}
 If Assumption \ref{fprimemin} is satisfied
\begin{itemize}
 \item System (\ref{total}) has a unique initial condition $x_0=x^f_0$ which is such that $x(T)=x_0$ if $u(t)=0$ for all times. It is the unique fixed point of
\[
\int_{x_0}^{x_0\mathrm e^{r(T-\overline{T})}}\frac{1}{f(\xi)-r\xi}\>\mathrm d\xi=\overline{T}\;.
\]
\item This value $x^f_0$ is such that, with $u(t)=0$ at all times, $x(T)>x_0$ for all $x_0<x^f_0$, and $x(T)<x_0$ for all $x_0>x^f_0$ so that  $x_{0max}=x^f_0$.
\item If
\begin{equation}\label{x0max}
 f'(0)\overline{T}>(r+\bar{u})T\;,
\end{equation}
there is a unique $x_0=x_{0min}$ such that  $x(T)=x_0$ if $u(t)=\bar u$ for all times. Also $x_0$ solution of (\ref{OCP}) belongs to the interval $[x_{0min},x_{0max}]$. 
\end{itemize}
\end{prop}
All this is obtained by analyzing the dependency
\[
\int_{x_0}^{x_(T)\mathrm e^{r(T-\overline{T})}}\frac{1}{f(\xi)-r\xi}\>\mathrm d\xi=\overline{T}\;.
\]
where $x(T)$ is seen as a function of $x_0$. 

Finally, since $x_0<\bar x^0$ in the solution of (\ref{OCP}), $x(t)<\bar
  x^0$ at all times. Indeed, in the bright phase, $x(t)$ cannot go through
  $\bar x^0$ since the choice of $u$ that maximizes $\dot x$ is $u=0$ that
  simply imposes convergence in infinite time toward $\bar x^0$; in the dark
  phase, $\dot x<0$ which also prevents $x(t)$ from going through $\bar x^0$. 

\subsection{Maximum principle}

In this section, we will show that the solution of (\ref{OCP}) can have one of
three patterns. In order to solve problem (\ref{OCP}), we will use
Pontryagin's Maximum Principle (PMP, \cite{Pont}) in looking for a control law
maximizing the Hamiltonian  
\[
H(x,u,\lambda,t)\triangleq \lambda\left(f(x)h(t)-rx-ux\right)+ux\;,
\]
with the constraint
\begin{equation}\label{x-lambda}
\left\{\begin{array}{lll}
\dot x&=& f(x)h(t)-rx-ux\;,\\
\dot \lambda&=&\lambda\left(-f'(x)h(t)+r+u\right)-u\;.
\end{array}\right.
\end{equation}
In addition, we should add the constraint
\[
\lambda(T)=\lambda(0)\;.
\]
as shown in \cite{Gilbert}. We see from the form of the Hamiltonian that  
\[
\frac{\partial H}{\partial u}=(1-\lambda)x\;,
\]
so that, when $\lambda>1$, we have $u=0$, when $\lambda<1$, we have $u=\bar{u}$, and when $\lambda=1$ over some time interval, intermediate singular control might be applied.

Of paramount importance in the proofs will be the constancy of the
Hamiltonian. Indeed, it is known that the Hamiltonian is constant along
optimal solutions as long as the time does not intervene into the dynamics or
the payoff. We then have that the Hamiltonian is constant in the interval
$[0,\overline{T})$ (with $h(t)=1$) and in the interval $(\overline{T},T]$
(with $h(t)=0$). The Hamiltonian presents a discontinuity at $\overline{T}$. 

We can first give some general statements about where and when switches can occur 
\begin{proposition}\label{prop1} If Assumption \ref{fprimemin} is satisfied then, in the solution of problem (\ref{OCP}) 
\begin{itemize}
\item[(i)] No switch from $u=0$ to $u=\bar{u}$ can take place in the dark phase;
\item[(ii)] Switches from $u=0$ to $u=\bar{u}$ in the bright phase take place with $x\leq x_\sigma$;
\item[(iii)] Switches from $u=\bar{u}$ to $u=0$ in the bright phase take place with $x\geq x_\sigma$.
\end{itemize}
\end{proposition}
\begin{proof}
All these results come from the analysis of $\dot \lambda$ at the switching instant. Indeed, at a switching instant, we have $\lambda=1$ so that:
\begin{equation}\label{atswitch}
\dot \lambda=\lambda\left(
-f'(x)h(t)+r\right)\;.
\end{equation}
The form of $\frac{\partial H}{\partial u}$ indicates that a switch from $u=0$ to $u=\bar{u}$ (resp. from $u=\bar{u}$ to $u=0$) can only occur if $\dot \lambda\leq 0$ (resp. $\dot \lambda\geq 0$).

In the dark phase, (\ref{atswitch}) becomes $\dot \lambda=r\lambda>0$; no switch from $u=0$ to $u=\bar{u}$ can therefore take place (this shows (i)).

In the bright phase, $\dot \lambda\leq0$ at a switching instant if $f'(x)\geq r$, which is equivalent to having $x\leq x_\sigma$; this is therefore a condition for a switch from $u=0$ to $u=\bar{u}$ (this shows (ii)).

Conversely, in the bright phase, $\dot \lambda\geq 0$ at a switching instant if $f'(x)\leq r$, which is equivalent to having $x\geq x_\sigma$; this is therefore a condition for a switch from $u=\bar{u}$ to $u=0$ (this shows (iii)).

\end{proof}

In the following, we will propose candidate solutions to the PMP by making various hypotheses on the value of $\lambda(0)=\lambda_0$.

\begin{theorem}\label{thm1}  If Assumption \ref{fprimemin} is satisfied then three forms of solutions of problem (\ref{OCP}) are possible:
\begin{itemize}
\item Bang-bang with $u(0)=0$, a single switch to $u(t)=\bar{u}$ taking place strictly before $\overline{T}$ and  a single switch back to $u(t)=0$ taking place strictly after $\overline{T}$;
\item Bang-singular-bang with $u(0)=0$, a switch to $u(t)=u_\sigma$ taking
  place first strictly before $\overline{T}$, a single switch to
  $u(t)=\bar{u}$ also taking place strictly before $\overline{T}$ and  a
  single switch back to $u(t)=0$ taking place strictly after $\overline{T}$. 
\item Constant control at $u(t)=\bar{u}$
\end{itemize}
In the first two cases, the switch back to $u(t)=0$ takes place with $x$ strictly smaller than it was at the moment of switch to $u(t)=\bar{u}$.
\end{theorem} 
The proof of this result is detailed in appendix and is obtained by considering all possble situations that are in concordance with Proposition (\ref{prop1}) and eventually eliminating all possibilities but the three cases detailed in Theorem  \ref{thm1}.

\subsection{Synthesis: the three possible optimal solutions}

Without needing an explicit form of $f(x)$, we have been able to obtain the
qualitative form of the optimal solution analytically. In the bang-bang case,
it is made of four phases:  

\begin{solution}{Bang-bang}\label{sol1}
\begin{enumerate}
\item Growth with a closed photobioreactor until a sufficient biomass level is
  reached ($u=0$, $\dot x>0$, $\lambda>1$, $\dot \lambda<0$);
\item Maximal harvesting of the photobioreactor with simultaneous growth ($u=\bar u$, $\dot x$ not determined, $\lambda<1$, $\dot \lambda<0$);
\item Maximal harvesting of the photobioreactor with no growth until a low level of
  biomass is reached ($u=\bar u$, $\dot x < 0$, $\lambda<1$, $\dot \lambda>0$);
\item Passive photobioreactor: no harvesting, no growth, only respiration ($u=0$, $\dot x < 0$, $\lambda>1$, $\dot \lambda>0$).
\end{enumerate}
\end{solution}
The first two phases take place in the presence of light, the other two in the dark. In phase~3, harvesting of as much biomass produced in the light phase as possible is continued while not going below the level where  the residual biomass left is sufficient to efficiently start again the next day.

If the optimal solution contains a singular phase, the analytical approach has helped us to identify the qualitative form of the optimal productivity solution. It now contains five phases:
\begin{solution}{Bang-singular-bang}\label{sol2}
\begin{enumerate}
\item Growth with a closed photobioreactor until the $x_\sigma$ biomass level is reached ($u=0$, $x<x_{\sigma}$, $\dot x>0$, $\lambda>1$, $\dot \lambda<0$);
\item Maximal equilibrium productivity rate on the singular arc ($u=u_\sigma$, $x=x_{\sigma}$, $\dot x=0$, $\lambda=1$, $\dot \lambda=0$);
\item Maximal harvesting of the photobioreactor with simultaneous growth ($u=\bar u$, $x<x_{\sigma}$, $\dot x<0$, $\lambda<1$, $\dot \lambda<0$);
\item Maximal harvesting of the photobioreactor with no growth until a low level of biomass is reached ($u=\bar u$, $x<x_{\sigma}$, $\dot x<0$, $\lambda<1$, $\dot \lambda>0$);
\item Passive photobioreactor: no harvesting, no growth, only respiration ($u=0$, $x<x_{\sigma}$, $\dot x<0$, $\lambda>1$, $\dot \lambda>0$);
\end{enumerate}
\end{solution} 

For this form of solution, we see that maximal instantaneous productivity is
achieved during the whole second phase, when the singular solution occurs
(Figure 3B). This solution with a singular form then seems to be naturally the
most efficient one. It can however not always be achieved for two reasons:  
\begin{itemize}
\item if (\ref{mumax}) is not satisfied, $u_\sigma>\bar{u}$, so that it is not
  an admissible control. The solution should then be bang-bang and the
  application of $u=\bar{u}$ has the same role as $u_\sigma$ in the solution
  with a singular arc since $u=\bar{u}$ is then the optimal solution to the
  instantaneous productivity optimization problem (Figure 3C). 
\item if growth is not sufficiently stronger than respiration (with (\ref{mumin}) satisfied however), a bang-bang solution that stays below $x_\sigma$ is optimal, because there is not enough time for it to reach $x_\sigma$ (Figure 3A).
\end{itemize}

The only remaining solution is the one with $u$ constant at $\bar u$ and is characterized by 
\begin{solution}{$u=\bar u$}\label{sol3}
\begin{enumerate}
\item Maximal harvesting of the photobioreactor with simultaneous growth ($u=\bar u$, $x<x_{\sigma}$, $\dot x>0$, $\lambda<1$, $\dot \lambda<0$);
\item Maximal harvesting of the photobioreactor with no growth ($u=\bar u$, $x<x_{\sigma}$, $\dot x<0$, $\lambda<1$, $\dot \lambda>0$);
\end{enumerate}
\end{solution} 
Since this solution relies on the satisfaction of condition (\ref{x0max}), it is clear that it could mainly occur if the actuator has been under-dimensioned and when the dark phase does not last too long.

\section{Discussion}\label{sec:large}

\subsection{Existence and uniqueness of the optimal solution}

Existence of an optimal solution is obvious since the achieved yield for
solutions satisfying  $x(t)=x(0)$ is bounded between $0$ (obtained for
$x_0=x_{0max}$ and $u=0$) and $\bar x^0\bar u A T$ (not achievable but an
upper-bound nonetheless because $x(t)\leq \bar x^0$ and $u(t)\leq \bar u$ at
all times). The productivity level therefore has a finite supremum, which
translates into a maximum since the set of definition of the $u(t)$ control
laws bounded by $[0,\bar u]$ is closed. The optimal control problem therefore has
a solution. 

We have also carried out a tedious analysis of the impossibility of existence
of two different solutions of the maximum principle by showing that, once one
solution has been evidenced, variations of the switching times 
cannot produce a second solution of the PMP satisfying both periodicity
conditions (on $x$ and $\lambda$). Once we have found a 
solution of the PMP, it is therefore optimal. This uniqueness result is not
crucial in the following discussions. Therefore, we do not detail them. 

\subsection{Large but limited $\bar u$}

Too small an upper-bound $\bar u$ for $u(t)$ gives rise to two kinds of
optimal solutions where the constraint on the control really limits the
productivity. First, when $\bar u<u_\sigma$, there could exist optimal
solutions that go through $x_\sigma$ but cannot stay at this optimal level
because the required control value is not admissible. Secondly, an optimal
solution (denoted Solution 3) could require the actuator to always be open.
In order to prevent the first case, it suffices to design the actuator and
chemostat so that $\bar u>u_\sigma$; note however that this does  not imply
that the solution of (\ref{OCP}) goes through a singular phase: growth could
very well be too slow for the solution to reach the $x_\sigma$ level during
the interval $[0,\bar{T}]$.  In the second case, the constant control
$u(t)=\bar u$ for all times also indicates that the actuator is not strong
enough since the only way the dilution can be efficient enough is by being
active during the whole dark phase; a solution where the  
dilution succeeds in harvesting the chemostat at the beginning of the night is
certainly to be preferred since it prevents respiration from taking too big a
role. Condition (\ref{x0max}) should then not be satisfied so that no periodic
solution with $u=\bar u$ exists. In this Section, we then should have that 
\[
  \bar u>\max\left(\frac{f\left[(f')^{-1}(r)\right]}{(f')^{-1}(r)}-r,\frac{f'(0)\overline{T}}{T}-r\right)\;.
\]
It is however convenient to remember from Remark~1 that $\frac{f\left[(f')^{-1}(r)\right]}{(f')^{-1}(r)}<f'(0)$, and that we trivially have $\frac{f'(0)\overline{T}}{T}<f'(0)$. We will therefore make the simpler hypothesis:
\begin{equation}\label{bigubar}
 \bar u>f'(0)-r\;.
\end{equation}
Note first that this implies that, when applying  $u=\bar u$, $\dot x=f(x)-rx-\bar u x<f'(0)x-rx-\bar u x<0$ so that $\bar x^{\bar u}=0$ and the biomass is always decreasing when the maximal dilution is applied.

When (\ref{bigubar}) is satisfied, only two forms of solutions of the PMP are possible: the bang-bang solution that never reaches the $x_\sigma$ level because the net biomass growth is too weak (respiration included) and the bang-singular-bang solution.

\subsection{Unconstrained dilution rate}

When considering that $u$ can be unbounded, there is the possibility for $\delta$ impulses to occur in the solution of the PMP. We will not give any further mathematical developments, but things readily seem clear. 

In the limit, bang-singular-bang solutions would have the following form: the
chemostat would be closed in the dark phase and at the beginning of the light phase until
$x=x_\sigma$ is reached. The solution would stay on $x=x_\sigma$ until
$t=\overline{T}$ is reached. Any earlier impulse would force the solution to
have more bang phases than possible, as demonstrated in the proof of Theorem \ref{thm1}. At $t=\overline{T}$,
a Dirac impulse is applied to bring $\lambda$ to $1$ so that $u=0$ is then
applied during the whole dark phase. The reactor then has three modes: a batch
mode when $x$ is different from $x_\sigma$, a continuous mode, on the singular
arc and an instantaneous harvest at the transition between the bright and dark
phases. The harvest consists in instantaneously replacing the medium with
biomass-free, but substrate-rich, medium, to bring the biomass concentration
to the mandated level.   

The bang-bang solutions that stay below $x_\sigma$ here become pure batch solutions with instantaneous harvest, the reactor being always closed except at $t=\overline{T}$ (where $\lambda$ is brought to $1$ through an impulse). 

We see from this analysis that, depending on the chosen microalgae and chemostat design, the optimal productivity is either obtained in batch mode or through the introduction of some continuous mode between batch phases.

\section{Simulations}\label{sec:simulations}

\subsection{{\it Isochrysis galbana}}\label{sub:low}

We will now show the forms of Solutions 1,2, and 3 in the $(t,x)$ space. For that, we start with a dynamical model (\ref{reduced-monod-bang}) for the growth of {\it Isochrysis galbana} with the parameters taken as in \cite{Masci2010}.
\begin{table}
\begin{center}$ \begin{array}{|c|c|c|}
\hline
  \mbox{Parameter}&\mbox{Value}&\mbox{Units}\\
\hline
\hline
\bar\mu & 1.7 &day^{-1}\\
a&0.5&m^2/g[C]\\
\bar I_0&1500& \mu \mbox{mol quanta}m^{-2}s\\
K_I&20&\mu \mbox{mol quanta}m^{-2}s\\
r&0.07 &day^{-1}\\
\overline{T}&0.5&day\\
\bar u&2&day^{-1}\\
\hline
 \end{array}$
\end{center}
\caption{Growth and bioreactor parameters for {\it Isochrysis galbana}}\label{Galbana}
\end{table}
With such parameters, the critical values $x_\sigma=14.93$ and
$u_\sigma=0.9066$ are easily computed, and the optimal solution is represented
by the blue curves in Figure~\ref{fig_case1}. It presents the
bang-singular-bang structure of Solution 2 with $u=0$ until time $t=0.282$
followed by $u=u_\sigma$ until $t=0.420$, $u=\bar u$ until $t=0.584$ followed
by $u=0$. The corresponding daily surfacic productivity is then $6.33
g[C]/m^2$ for a total cumulated flow $\int_0^1 u(\tau)\>\mathrm d\tau$ equal
to $0.453$, that is $45\%$ of the medium has been renewed during the $24$ hours. We then considered the application of a
constant control during the $24$ hours and optimized the level of this control
numerically. The optimum was achieved for $\hat u=0.461$, which yields a daily
productivity equal to $6.26 g[C]/m^2$ and a cumulated flow equal to $0,461$
also. Three things need to be noticed from this comparison: (i) both optimal
solutions are quite different in Figure~\ref{fig_case1} though the $x$ values
stay in the same  
range; (ii) the productivity increase generated by the optimal solution is
very weak ($1.11\%$); (iii) the total flow required to attain both optimal
solutions are very similar. In fact, the fact that the improvement of the
productivity is small is not surprising: the necessity of shutting down the
chemostat at night is linked to the respiration that would consume the
biomass. In the present case, the respiration is weak so that, during one
night, only $3.44\%$ of the biomass is consumed. This phenomenon is here
marginal so that the optimal control approach developed to limit it provides
little gain and what really matters is the total flow that goes through the
photobioreactor.   

\begin{figure}[t]
\centerline{\includegraphics[width=0.7\columnwidth]{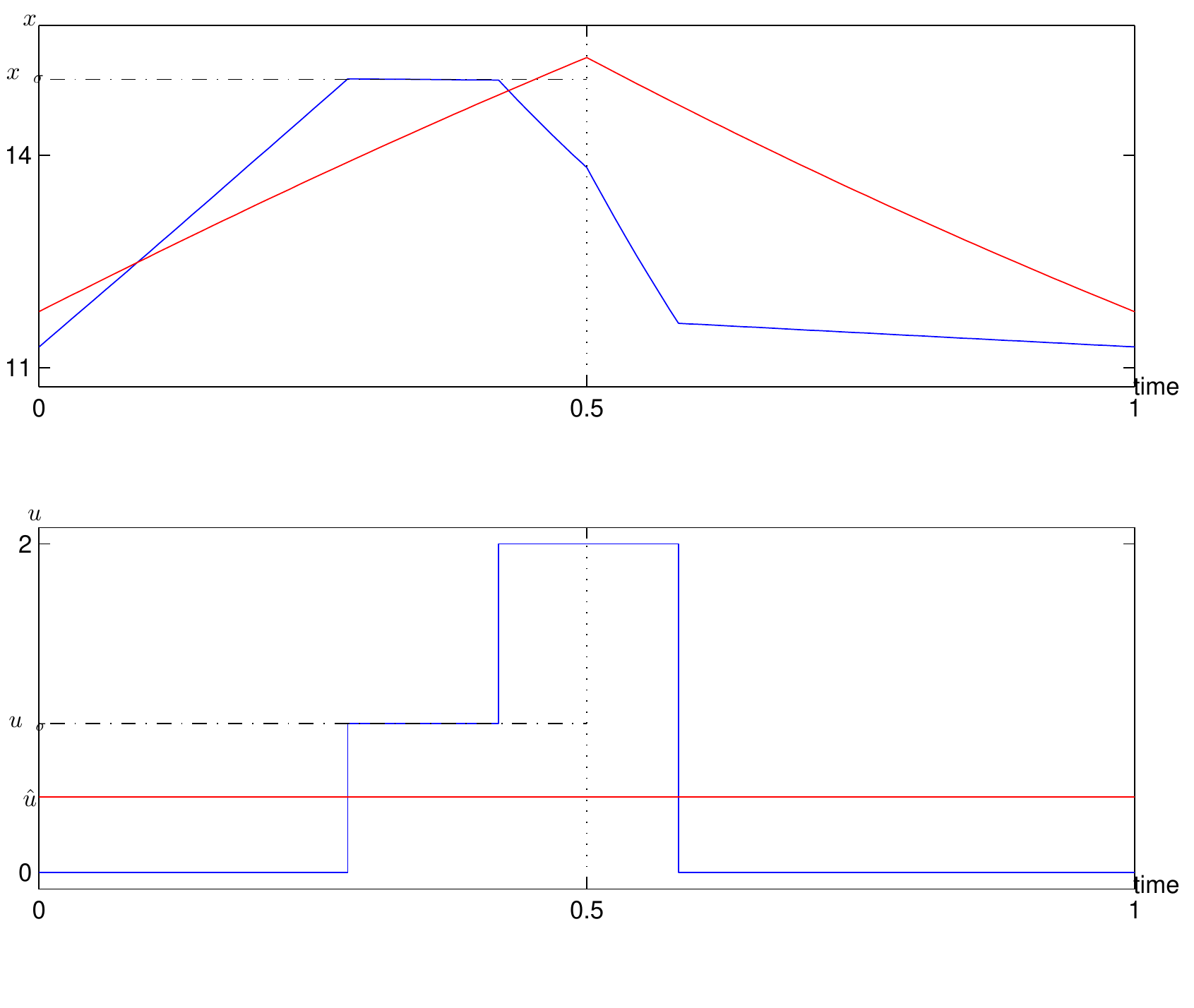}}
\caption{Bang-singular-bang optimal solution (in blue) confronted with the
  most productive constant dilution rate scenario (in red) for the microalgae {\it Isochrysis
  galbana} and the parameters of Table \ref{Galbana}. At the top is the
  evolution of the biomass and at the bottom that of the control. The black
  dash-dotted lines represent the values of $x_\sigma$ and $u_\sigma$
  respectively}\label{fig_case1} 
\end{figure}

A second case needs to be studied, that is the one that corresponds to the
case where $\bar u$ is smaller than $u_\sigma$, so that the singular phase is
not possible. This yields Solution 1, the bang-bang case, when we take $\bar
u=0.8$. However, we expect that the productivity would not be degraded much
since we should be able to do better than the aforementioned case with $\hat
u=0.461$, which is an admissible solution for the present control problem, and
which gives a productivity that is  barely below the optimal one with $\bar
u=2$. It is indeed the case since, with switching times at $t=0.222$ and
$t=0.790$ a productivity of $6.30 g[C]/m^2$ is achieved with a total flow of
$0.457$ still similar to the one observed in both previous cases. 

If we now reduce $\bar u$ to $0.1$, a significant performance degradation
should be expected since there is an upper-bound on the total flow. In fact,
we obtained numerically that the optimal control here yields a constant
control at $u=0.1$ and a productivity of $4.35 g[C]/m^2$ which is $68.7 \%$ of
the one obtained when $\bar u=2$.    

\subsection{Importance of the respiration}\label{sub:high}

In this section we explore the impact of a large value for parameter $r$, which can be due to increased respiration, or to a high mortality as often is the case in high density cultures (however, we will stick to the respiration terminology). 
If we now consider a species that has all the characteristics of {\it Isochrysis
galbana} except that it has a very large $r$ equal to $0.7day^{-1}$, we
expect the optimal strategy to have a much bigger impact on the
outcome. Indeed, this respiration has much more importance at night than in
the previous case since it consumes $29.5\%$ of the biomass, hence the
importance of limiting the biomass level at night. We see in
Figure~\ref{fig_case3} that the optimal solution is here bang-bang with a
short opening window at the end of the day and at the beginning of the night
to harvest the produced biomass ($u=\bar u$ for $t\,\in[0.479,0.527]$). The
optimal production is here $0.607 g[C]/m^2$ for a total flow of $0.096$, that
is a very little daily medium renewal while the best constant control $\hat
u=0.095$ yields $0.519 g[C]/m^2$. We see here that the daily total flows are
again almost equivalent but that the productivity is here improved by $17\%$
through the bang-bang approach, which is far from being  
negligible, especially since it is made with almost an exact same hydraulic
effort as the constant dilution strategy. Though the $x(t)$ solutions in
Figure~\ref{fig_case3} both look similar, the larger population at night with
the constant control strategy explains why there is more respiration when the
control is constant than when it is bang-bang, hence less productivity.  
\begin{figure}[t]
\centerline{\includegraphics[width=0.7\columnwidth]{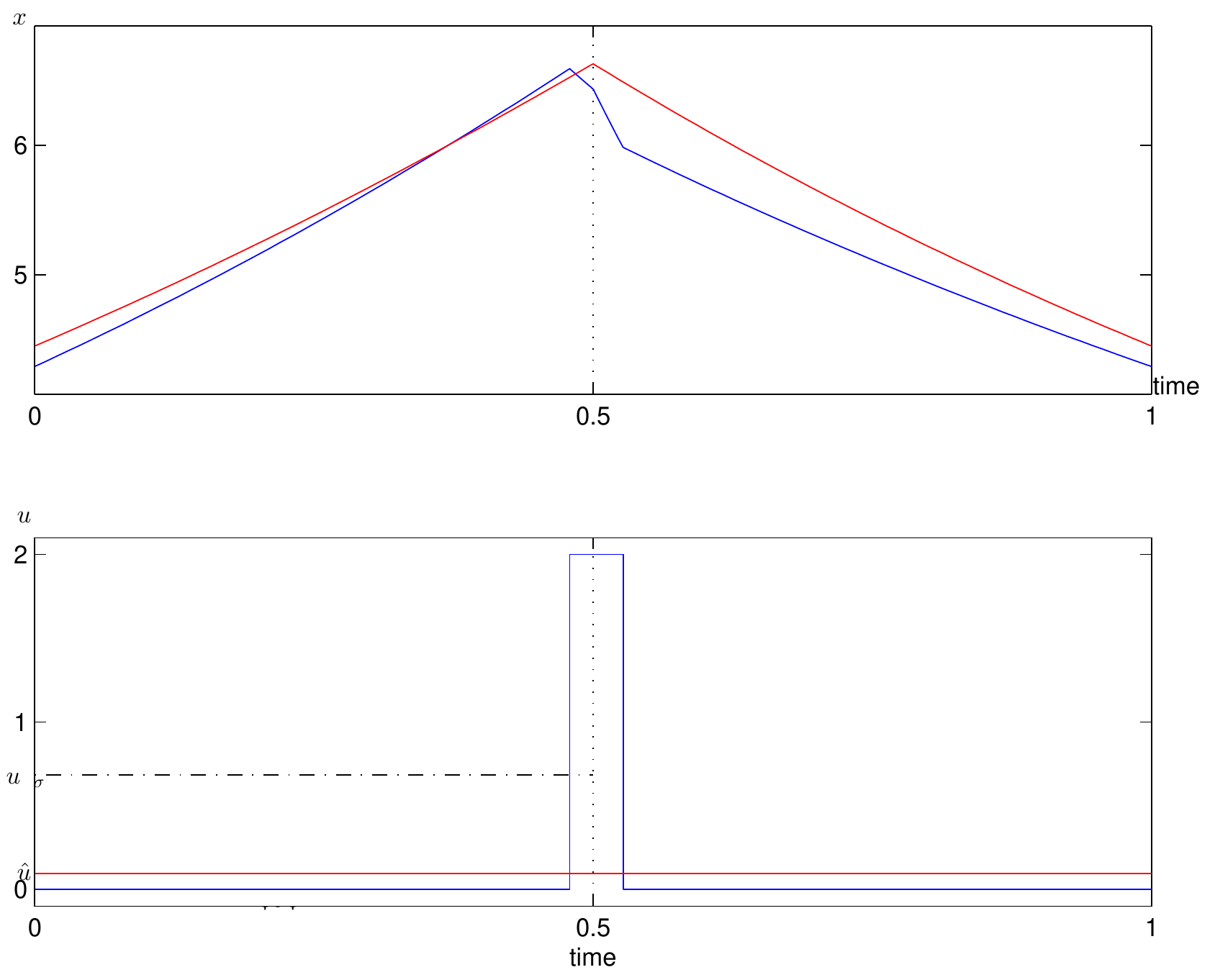}}
\caption{Bang-bang optimal solution (in blue) confronted with the most productive constant dilution rate scenario (in red) for a high-respiration species ($r=0.7$) and $\bar u=2$}\label{fig_case3}
\end{figure}

\subsection{Bifurcation analysis}\label{sub:bif}

 \begin{figure}[t]
 \centerline{\includegraphics[width=0.8\columnwidth]{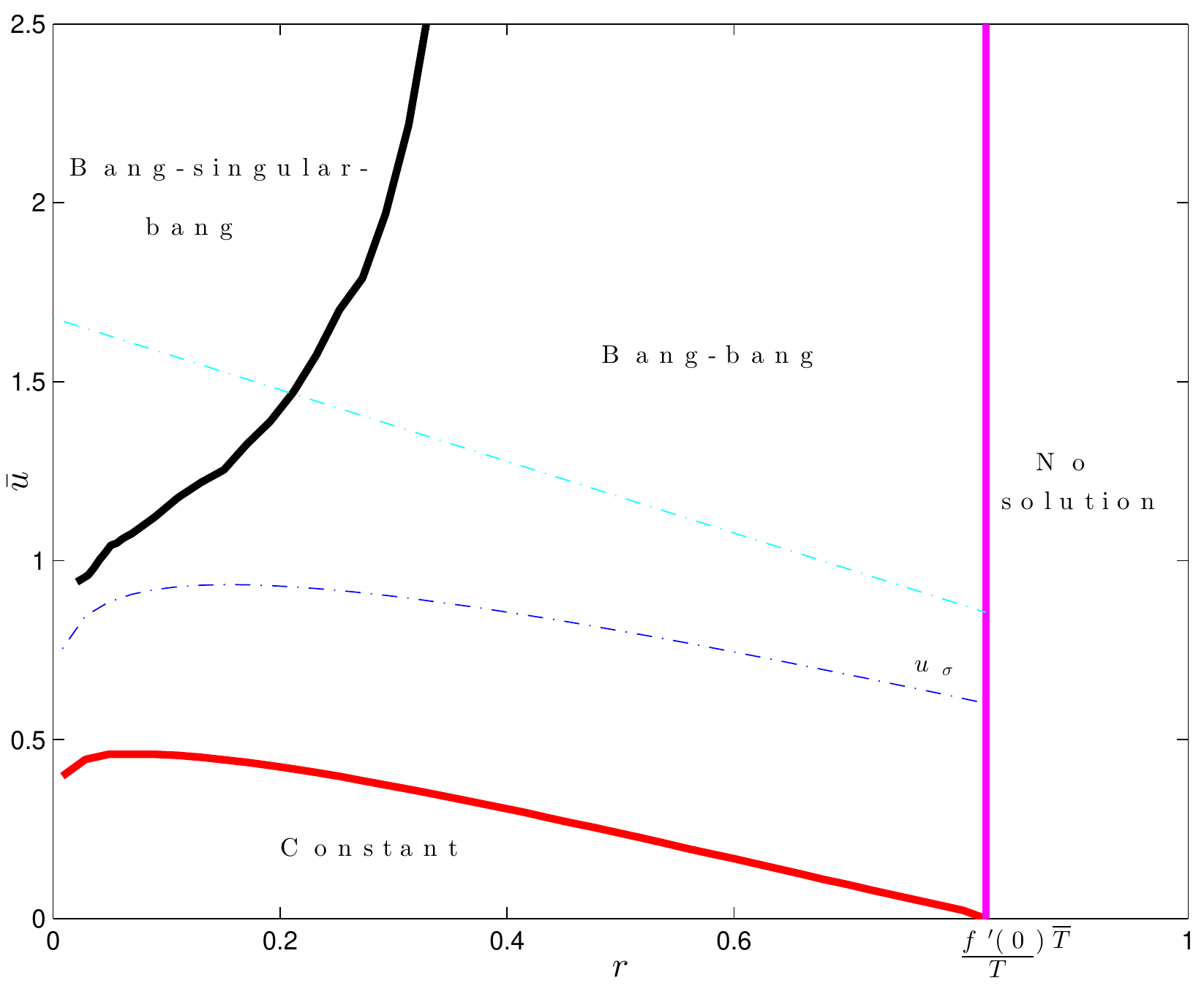}}
 \caption{Bifurcation diagram defining the four $(r,\bar u)$-parametric
   regions where the different patterns of optimal solutions are encountered:
   bang-singular-bang control (above the black line), bang-bang control
   (between the black, red, and magenta lines, constant control (below the red
   line) and no optimal solution (on the right of the magenta line)
 }\label{bifurcation_rubar} 
\end{figure}

In this section, we study more quantitatively the outcome of sections \ref{sub:low}
and \ref{sub:high}. We draw 2D-bifurcation figures for
the parameters $r$ and $\bar u$  (the other parameters being fixed at the
values of Table~\ref{Galbana}). We build a bifurcation diagram for these two
parameters by identifying in which regions the different optimal solution
patterns appear; these regions are delimited by solid lines. 

We first see in Figure~\ref{bifurcation_rubar} that, when $r>\frac{f'(0)\bar
  T}{T}$, no solution is possible because condition (\ref{mumin}) is not
satisfied. Only the wash-out of the photobioreactor can occur, whatever the
control strategy. In the region below the red curve, the optimal solution is a
constant control $u(t)=\bar u$ during the whole day and night;  this was
expected for these small values of $\bar u$ since the photobioreactor produces
a quantity of biomass during the day that cannot be taken out if the actuator
is not always open. The region where the solution is bang-bang is split into
two by the $\bar u=u_\sigma(r)$ curve (dashed-dotted blue line of Figure
\ref{bifurcation_rubar}): when $\bar u<u_\sigma(r)$, no singular phase is
possible and when $\bar u>u_\sigma(r)$ in that region, the singular phase is
theoretically possible but does not occur because the biomass does not reach
the $x_\sigma$ level. Finally, we see that the bang-singular-bang phase is
concentrated in a region where $r$ is  
small and $\bar u$ large: the former allows for a reduced natural decrease of
the biomass concentration, so that it does not need a lot of effort to be
brought back up to $x_\sigma$ at the beginning of the light phase; the latter
allows for a very short phase of maximal harvesting so that it leaves plenty
of time for a singular phase. Finally, we illustrated condition
(\ref{bigubar}) by a dashed-dotted cyan line; this figure confirms what we had
evidenced earlier: when $\bar u$ is above that line the optimal control does
not suffer from control limitations that intrinsically prevent a singular
phase (because $\bar u<u_\sigma$) or forces the optimal control to be
constant.

\subsection{Importance of the open-reactor phase}

In the high-respiration case, we have identified a short window of opening of
the bioreactor around the end of the day and the beginning of the night. In
this subsection, we will show a simulation that evidences the fact that the
main characteristics of this window is its length and not so much the exact
time at which it takes place. We have considered the case where $r=0.7$ and
computed the value of the productivity for switching times $t_1$ from $u=0$ to
$\bar u$ between $0.44$ and $0.5$ and $t_2$ from $\bar u$ to $0$ between $0.5$
and $0.56$. The productivity is then illustrated by different color levels in
Figure~\ref{contour_prod}, the dark blue corresponding to zero (wash out of
the reactor or almost closed reactor) and the purple to values above $0.6$. A
definite pattern appears on this figure: the productivity level is roughly
constant along lines of the form $t_2=\delta+t_1$. The productivity level then
mainly depends on $t_2-t_1$, that is the opening duration of the reactor or,
equivalently, the  
total flow that goes through the reactor.
\begin{figure}[t]
\centerline{\includegraphics[width=0.7\columnwidth]{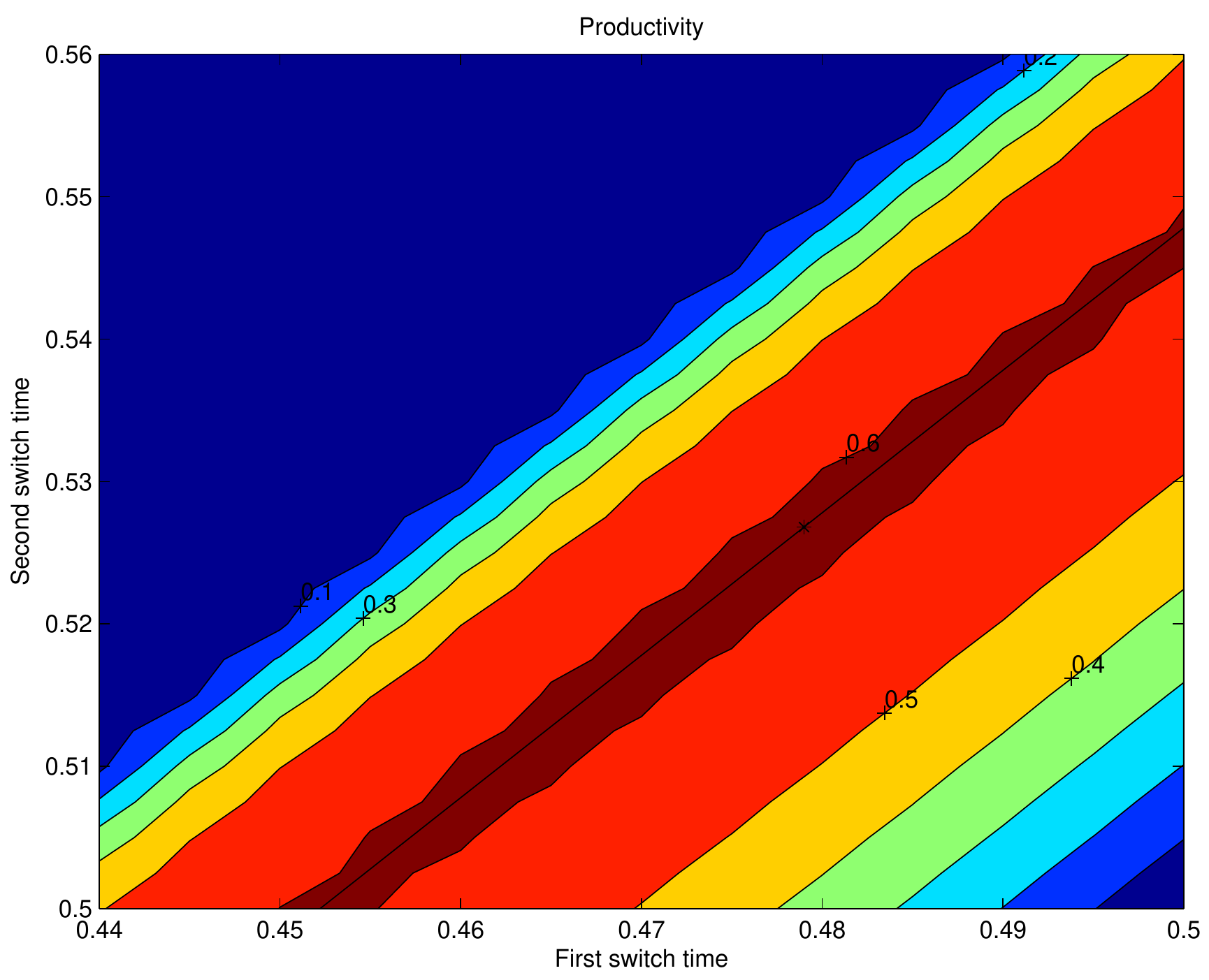}}
\caption{Productivity level contours for bang-bang solutions with the first switching time in abscissa and the second switching time in ordinate. The optimal productivity level computed in the previous section is obtained at the black star.}\label{contour_prod}
\end{figure}

\subsection{Near-optimal strategies}\label{sec:near}

We have seen that the daily flow that goes through the chemostat has great
importance for the productivity level of a solution. In order to confirm
that, we propose a strategy that is of the bang-bang type and consists in
having the dilution equal to $\bar u$ in the interval
$[\overline{T}-\displaystyle\frac{\tilde u}{2\bar
  u},\overline{T}+\displaystyle\frac{\tilde u}{2\bar u} ]$ and $0$ outside of
this interval. That way, the total flow that goes through the reactor is equal
to $\tilde u$, so that we will be able to compare the obtained productivity
level between that strategy and constant control strategies that have the same
daily total flow. We did the computations for $\bar u=2$ and values of $\tilde
u$ that did not lead to the wash-out of the reactor for both species of
subsections \ref{sub:low} (on the left of Figure \ref{near}) and
\ref{sub:high} (on the right of Figure \ref{near}). We see that, in both
cases, the constant control (in red) yields a less productive process that the
corresponding  bang-bang strategy. This is especially true in the
high-respiration case, but is 
also valid in the low respiration case, so that, when the exact optimal
control law is not computed, it is advisable to choose a bang-bang one rather
than a constant control. This strategy is stronly advisable since the optimal
productivity level is represented through a dotted level in both figures and
the near-optimal strategy achieves it almost in both cases for an appropriate
daily flow. This was expected in the high respiration case where the optimal
solution is bang-bang, because we had evidenced in Figure \ref{contour_prod}
that the actual timing of the beginning of the max-control window had little
influence on the productivity level, but is also the case in the
low-respiration case where the optimal solution is bang-singular-bang; even in
this case, our proposed near-optimal strategy can almost achieve the optimal
control level.  This last property might however not hold for a high
respiration species whose optimal pattern is bang-singular-bang: in that case,
the corresponding best bang-bang near-optimal strategy might lead to large
values of $x(t)$ hence more respiration, at the time where the control should
be singular.

\begin{figure}[t]
\centerline{\includegraphics[width=0.45\columnwidth]{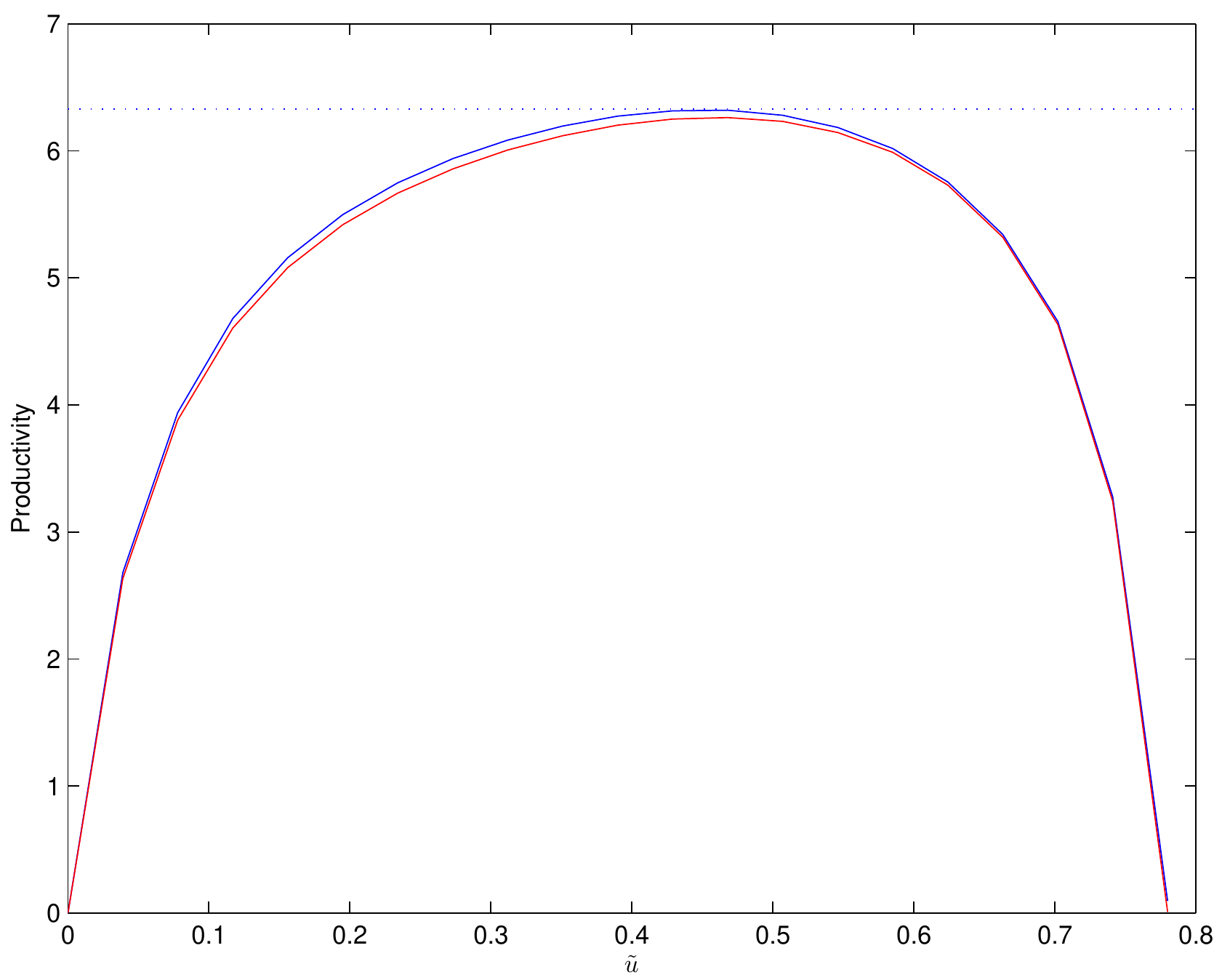}\includegraphics[width=0.45\columnwidth]{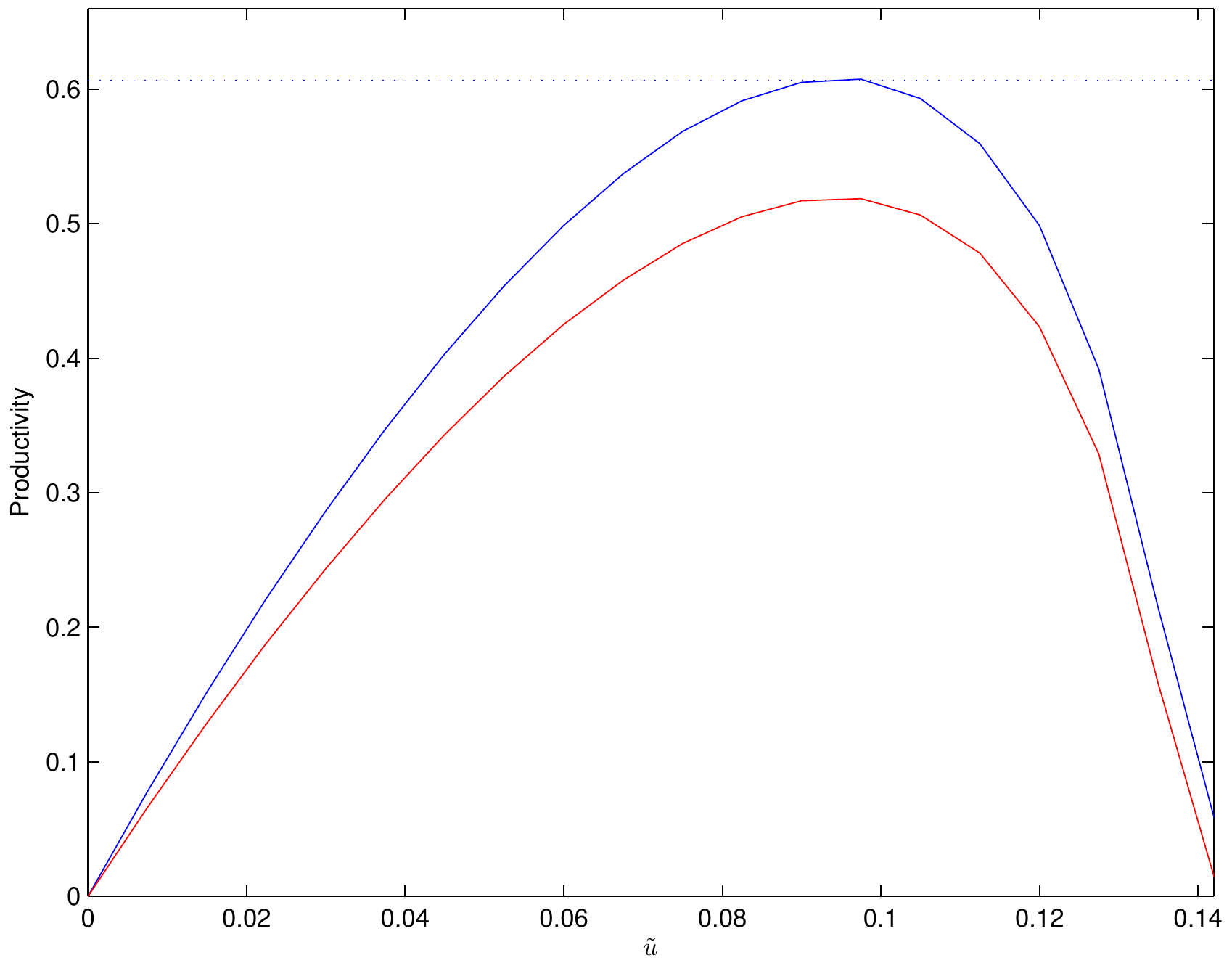}}
\caption{Productivity levels of species of subsections \ref{sub:low} (on the
  left) and \ref{sub:high} (on the right ): in red, the productivity attained
  with the constant control $u=\tilde u$ and in blue the one with the
  near-optimal strategy of Section \ref{sec:near}. The dashed-lines represent
  the optimal productivity levels with $\bar u=2$}\label{near} 
\end{figure}



\subsection{Beyond the photobioreactor}

Photobioreactors are not the only ecological systems that undergo natural  periodic forcing. We will now derive an example that relies more on seasonality and that is built on the famous fishing-stock model  and the Maximum-Sustainable-Yield (MSY, \cite{Clark}).  The question here is therefore to determine how fishermen can best exploit a fishing stock while allowing it to survive. When basing the study on the logistic growth model
\[
\dot x=\alpha x\left(1-\frac{x}{K}\right)-qEx
\]
where $x$ is the size of the stock, $E$ is the fishing effort and $q$ is the fish  catchability, it is determined that the maximization of the number of caught fishes over the long-run is obtained by keeping the stock at $x=\frac{K}{2}$ with $E=\frac{1}{2q}$. These are our $x_\sigma$ and $u_\sigma$ values. If we now consider that the fishing stock has a limited growing season (of length $\bar T$) during which it satisfies the logistic growth) and a season (of length $T-\bar T$) where only natural mortality and predation take place, we can set ourselves in the setting of the present paper. Indeed we can then define $f(x)-rx=\alpha x\left(1-\frac{x}{K}\right)$ with $r$ the mortality rate during the non-growing season. We then define $u=qE$.  Taking $\alpha=6$, $K=10$, $r=1$, $\bar T=0.2$, $T=1$ and $\bar u=2$, we first noticed on Figure \ref{fig_fishing} (black curve) that, in the absence of fishing, the fishing stock settles at a level that is way below $K$ and even below $\frac{K}{2}$, the value of the MSY, which shows that no singular phase is possible in the optimal solution. The bang-bang optimal solution is then computed (blue curve) with switching times at $t=0.188$ and $0.288$ and compared with the best constant harvesting solution (red curve). In both cases, the fishing stocks are very similar during the growing season, but the optimal harvesting method reduces it at the beginning season so that mortality does not have time to do a lot of damage and, in the end, it improves the total of caught fishes during the season by 37\%.
\begin{figure}[t]
\centerline{\includegraphics[width=0.7\columnwidth]{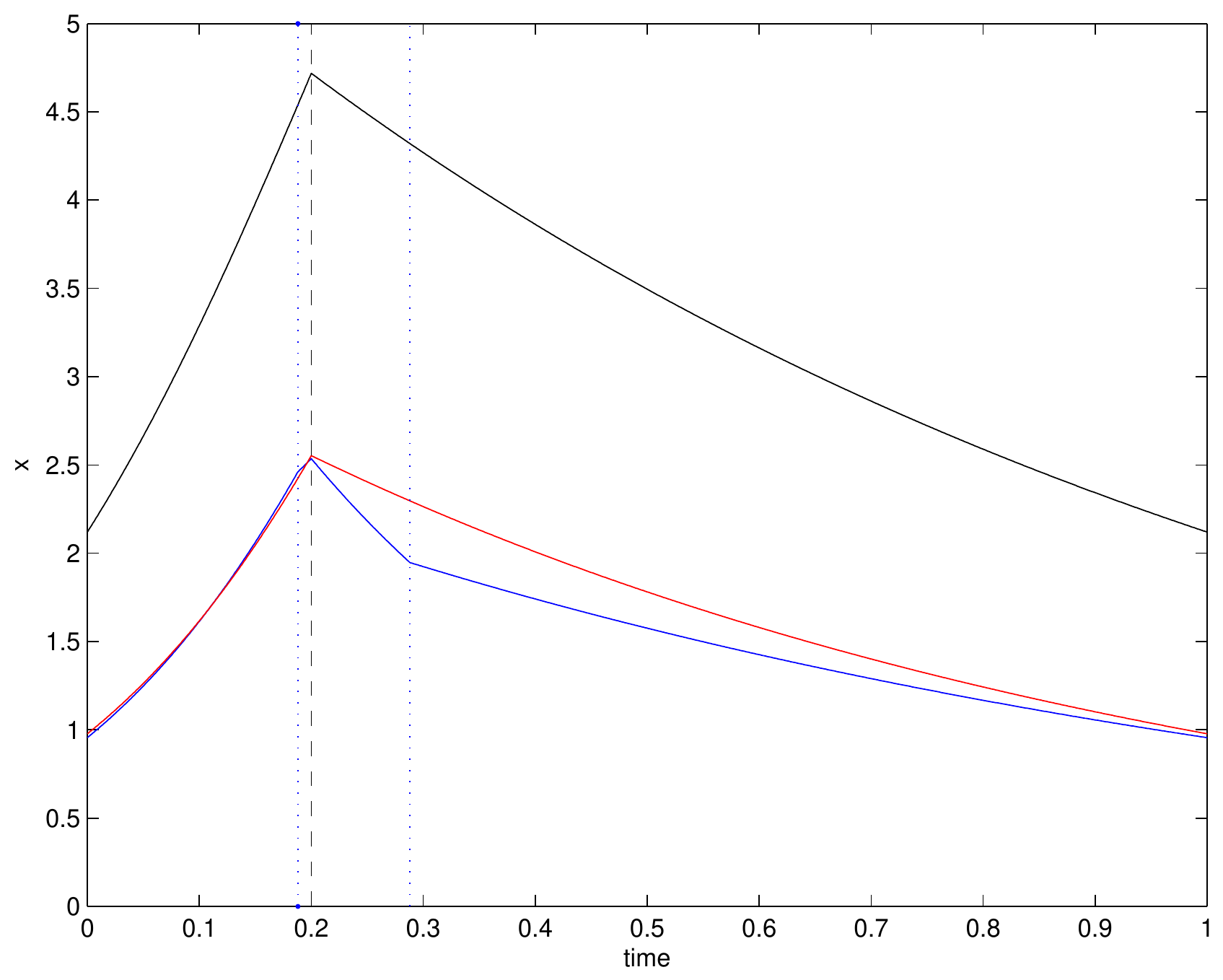}}
\caption{Comparison of the fishing stock when not fished (black curve), constantly fished (red curve) and using the bang-bang optimal control (blue curve).}\label{fig_fishing}
\end{figure}

\section{Conclusions}

In this paper, we have shown how the day-night constraint influences the
optimal control strategy that achieves maximal biomass daily productivity in a
photobioreactor and reduces the optimal productivity level. We have identified
three families of strategies that can achieve optimality: bang-singular-bang,
bang-bang, and constant maximal control; the first two are characterized by an
harvesting of the biomass at the end of the light phase and beginning of the
night to limit the negative effects of the respiration, while the latter leads
to permanent harvesting because the maximal dilution rate is too weak compared
with the growth rate of the biomass. These families of control strategies have
been built for a large set of nonlinear one-dimensional models with light-dark
phases so that they can be applied beyond their motivation in this paper: that
is a photobioreactor with Monod-like growth and Beer-Lambert light
attenuation. Through simulations, we have shown that the necessity of applying
an optimal control  
strategy strongly depends on the respiration rate: if the latter is weak,
constant control can achieve almost the same performance 
 since the night phase consumes little
biomass. However, we have also shown that a better choice is probably to apply
a bang-bang control law which, in the presented simulations, always yields
better productivity than comparable strategies with constant dilution. This is
particularly supported by the fact the a bang-bang law with a proper timing
can almost achieve the same productivity level as the optimal control law
developed in the present paper; also this better productivity is achieved with
a very similar cumulated effort as the one necessary for a constant dilution
rate.



\bibliographystyle{ieeetr}
\bibliography{biblio_photoB}

\appendix

\section{Proof of Theorem \ref{thm1}}

We will detail the different forms of solutions and show that no other solution than the one exposed in Theorem \ref{thm1} can occur. \\
{\bf Bang-bang with $\lambda_0>1$: } With $\lambda_0>1$, we have $u=0$ at
times $0$ and $T$. At least, a switch of $u$ from $0$ to $\bar{u}$ and back to
$0$ then needs to occur between time $0$ and $T$ (otherwise the payoff would
be $0$, since the argument of the integral that defines the productivity would
always be $0$). As we have seen 
in Proposition~\ref{prop1}-(i), the switch from $0$ to $\bar{u}$ cannot take place in the dark phase. For the solution that we
study, a switch then needs to take place at time $t_{0\bar{u}}$ in the $(0,\overline{T})$ interval and for $x(t_{0\bar{u}})=x_{0\bar{u}}<x_\sigma$ and $\lambda(t_{0\bar{u}})=1$ (Proposition~\ref{prop1}-(ii)).

The switch from $u=\bar{u}$ to $u=0$ then needs to occur at some time $t_{\bar{u}0}>t_{0\bar{u}}$ (and some value of $x(t_{\bar{u}0})=x_{\bar{u}0}$). Two possibilities then need to be considered:
\begin{itemize}
\item $t_{\bar{u}0}\leq\overline{T}$: from Proposition \ref{prop1}, we can
  deduce that $x_{\bar{u}0}\geq x_\sigma$. Applying $u=0$ then forces
  convergence toward $\bar x^0$ for $t_{\bar{u}0}\leq t\leq \overline{T}$, so
  that $x(t)$ increases, and so stays above $x_\sigma$ (so that
  $x(\overline{T})\triangleq x_{\overline{T}}\geq x_\sigma>x_0$). No switch
  back to $u=\bar{u}$ can then take place neither before $\overline{T}$
  (because of (ii) of Proposition \ref{prop1}) nor after $\overline{T}$
  (because of (i) of Proposition \ref{prop1}).  
\item $t_{\bar{u}0}>\overline{T}$: from Proposition \ref{prop1}-(i), we can
  deduce that no other switch takes place afterward. Two cases will then be
  considered: either $x_{0\bar{u}}\geq \bar x^{\bar{u}}$ so that $x(t)$
  decreases and $x_{0\bar{u}}\geq x_{\overline{T}}>x_{\bar{u}0}$ or
  $x_{0\bar{u}}< \bar x^{\bar{u}}$ so that $x_{0\bar{u}}< x_{\overline{T}}$,
  and we do not necessarily know if $x_{0\bar{u}}>x_{\bar{u}0}$ 
\end{itemize}
We have then shown that any bang-bang solution with $u(0)=0$  presents a
single switch to $u=\bar{u}$ (before $\overline{T}$) and a single switch to
$u=0$ (before or after $\overline{T}$). We are then left with two things to
show: first that $t_{\bar{u}0}\leq \overline{T}$ leads to a contradiction and
then that, when $x_{0\bar{u}}< \bar x^{\bar{u}}$, we indeed have
$x_{0\bar{u}}>x_{\bar{u}0}$. 

We will now show that the switch from $\bar{u}$ to $0$ needs to occur strictly
after $\overline{T}$. For that, we will consider that $t_{\bar{u}0}\leq
\overline{T}$ and show that this leads to a contradiction. Using the constancy
of the Hamiltonian in $[0,\overline{T})$, we have, by continuity of the
variables at time $\overline{T}$, that 
\[
\lambda_0\left(f(x_0)-rx_0\right)=\lambda_{\overline{T}}\left(f(x_{\overline{T}})-r x_{\overline{T}} \right)\;,
\]
and, in $(\overline{T},T]$, with $x(T)=x_0$ and $\lambda(T)=\lambda_0$:
\[
\lambda_0\left(-rx_0\right)=\lambda_{\overline{T}}\left(-r x_{\overline{T}} \right)\;.
\]
Taking the differences between the equalities, we obtain
\[
\lambda_0f(x_0)=\lambda_{\overline{T}}f(x_{\overline{T}})\;.
\]
Taking the quotient of the last two equalities then yields
\[
\frac{f(x_0)}{x_0}=\frac{f(x_{\overline{T}})}{x_{\overline{T}}}\;,
\]
with $x_{\overline{T}}\geq x_\sigma>x_0$. Remembering from Remark~1 that
$\frac{f(x)}{x}$ is a decreasing function, this leads to a contradiction
because $x_{\overline{T}}>x_0$ since the biomass only decreases at night. The
switch from $u=\bar{u}$ to $u=0$ then needs to take place with
$t_{\bar{u}0}>\overline{T}$. 

We will now show that, even in the case where $x_{0\bar{u}}<\bar x^{\bar{u}}$,
the second switch takes place with $x_{\bar{u}0}<x_{0\bar{u}}$. For that, we
first note that, since no switch from $u=\bar{u}$ to $u=0$ takes place in the
$[t_{0\bar{u}},\overline{T}]$ interval,
$f(x_{\overline{T}})-rx_{\overline{T}}>f(x_{0\bar{u}})-rx_{0\bar{u}}$. Indeed,
since $x(t)$ is increasing along the considered solutions from
$x_{0\bar{u}}<x_\sigma$, $f(x(t))-rx(t)$ first increases from
$f(x_{0\bar{u}})-rx_{0\bar{u}} $and then decreases (once $x(t)$ gets above
$x_\sigma$); if we had $f(x_{\overline{T}})-rx_{\overline{T}}\leq
f(x_{0\bar{u}})-rx_{0\bar{u}}$, there would be a time $\tilde t$ belonging to
the interval $(t_{0\bar{u}},\overline{T}]$ where $f(x(\tilde t))-rx(\tilde
t)=f(x_{0\bar{u}})-rx_{0\bar{u}}$. A switch then necessarily would have taken
place at that instant because of constancy of the Hamiltonian 
\[
\left(f(x_{0\bar{u}})-rx_{0\bar{u}}\right)=\lambda(\tilde t)\left(f(x(\tilde t))-(r+\bar{u}) x(\tilde t) \right)+\bar{u}x(\tilde t)
\]
results, using $f(x(\tilde t))-rx(\tilde t)=f(x_{0\bar{u}})-rx_{0\bar{u}}$, in
\[
(1-\lambda(\tilde t))\left(f(x(\tilde t))-rx(\tilde t)\right)=(1-\lambda(\tilde t))\bar{u}x(\tilde t)\;.
\]
This can only be achieved if $\lambda(\tilde t)=1$, which corresponds to a
switching instant (and a contradiction) or $\left(f(x(\tilde t))-rx(\tilde t)
\right)=x(\tilde t)$, which would mean that $x(\tilde t)=\bar x^{\bar{u}}$,
which also is impossible since, in that phase, $x(t)$ was converging in
infinite time toward $\bar x^{\bar{u}}$ (and $x_{0\bar{u}}<\bar
x^{\bar{u}}$). We then conclude that, as announced,
$f(x_{\overline{T}})-rx_{\overline{T}}>f(x_{0\bar{u}})-rx_{0\bar{u}}$.  

We then extensively use the constancy of the Hamiltonian in both the bright and the dark phase  to show that $x_{\bar{u}0}<x_{0\bar{u}}$. In the bright phase, we have 
\[
\lambda_0\left(f(x_0)-rx_0\right)=\left(f(x_{0\bar{u}})-rx_{0\bar{u}}\right)\;,
\]
and, in the dark phase
\[
\lambda_0\left(-rx_0\right)=\left(-rx_{\bar{u}0}\right)\;.
\]
Taking the difference between these equalities yields
\[
\lambda_0f(x_0)=\left(f(x_{0\bar{u}})-rx_{0\bar{u}}\right)+rx_{\bar{u}0}\;.
\]
Finally we note that, after $t_{\bar{u}0}$, the dynamics become
\[
\dot \lambda =r\lambda\;,
\]
so that, using $\lambda(t_{\bar{u}0})=1$ and $\lambda(T)=\lambda_0$, we have
\[
\lambda_0=\displaystyle \mathrm e^{r(T-t_{\bar{u}0})}\;,
\]
which yields
\[
f(x_{0})e^{r(T-t_{\bar{u}0})}-rx_{\bar{u}0}=f(x_{0\bar{u}})-rx_{0\bar{u}}\;.
\]
Concavity implies that $f(x_{0})e^{r(T-t_{\bar{u}0})}>f(x_{0}e^{r(T-t_{\bar{u}0})})=f(x_{\bar{u}0})$, so that the constancy of the Hamiltonian conditions can only be satisfied if 
\[
f(x_{\bar{u}0})-rx_{\bar{u}0}<f(x_{0\bar{u}})-rx_{0\bar{u}}\;.
\]
This can only be achieved for $x_{\bar{u}0}<x_{0\bar{u}}$, since $x_{\bar{u}0}<x_{\overline{T}}$ and $f(x_{\overline{T}})-rx_{\overline{T}}>f(x_{0\bar{u}})-rx_{0\bar{u}}>f(x_{\bar{u}0})-rx_{\bar{u}0}$.

{\bf Bang-singular-bang with $\lambda_0>1$: }

We will first look at what a singular arc could be. For that, we see that $\frac{\partial H}{\partial u}=(1-\lambda)x$ should be 0 over a time interval which, since a biomass level of $x=0$ does not make sense when optimizing the productivity, amounts to imposing $\lambda=1$. We then compute its time derivatives.
\[
\frac{d}{dt}\left(\lambda-1\right)_{\arrowvert \lambda=1}=-f'(x)h(t)+r\;.
\]
When $h(t)=0$, that is in the dark phase, no singular arc is thus
possible. When $h(t)=1$, this derivative is equal to zero when $x=x_\sigma$
defined in (\ref{xsigma}). The singular control is then the control that
maintains this equilibrium, that is $u_\sigma$ defined in (\ref{usigma}). This
control is positive thanks to (\ref{mumin}) 
but it is smaller than $\bar{u}$ only if 
\begin{equation}\label{mumax}
f\left[(f')^{-1}(r)\right]<(r+\bar{u})(f')^{-1}(r)\;.
\end{equation}
No singular control can exist otherwise. In fact, this implies that $\bar x^{\bar{u}}<x_\sigma<\bar x^0$. The case where $u_\sigma=\bar{u}$ has been handled through the bang-bang case.

When a singular branch appears in the optimal solution, it is locally optimal because the second order Kelley condition \cite{Kelley}:
$$
\frac{\partial}{\partial u}\left(\frac{\mathrm d^2}{\mathrm d\tau^2}\frac{\partial H}{\partial u}\right)=-f''(x_\sigma)\geq 0
$$
is satisfied on the singular arc because of the concavity. 

The construction of the solution is very similar to that in the purely
bang-bang case. Similarly, a first switch needs to  occur in the interval
$(0,\overline{T})$. This switch can be from $u=0$ to $u=\bar{u}$ or from $u=0$
to $u=u_\sigma$ and should occur with $x\leq x_\sigma$
(Proposition~\ref{prop1}-(ii)). In fact, if a switch first occurs to
$u=\bar{u}$, $x$ then converges toward $\bar x^{\bar{u}}$, so that it does not
go on or above $x_\sigma$, which prevents any other switch before
$\overline{T}$ (Proposition \ref{prop1}-(iii)). No singular arc can then
appear. These solutions have been handled earlier.  

In bang-singular-bang cases, a switch from $0$ to $u_\sigma$ then directly takes place once $\lambda=1$ at $(t_{0\sigma},x_\sigma)$. 

From there, $\lambda(t)=1$ and $x(t)=x_{\sigma}$ for some time. This could be
until $t_{\sigma 0}\leq\overline{T}$, followed directly by $u=0$ or the
singular arc could end at time $t_{\sigma \bar{u}}<\overline{T}$, where a
switch occurs toward $u=\bar{u}$; note that the strict inequality is due to
the fact that, if we had $t_{\sigma \bar{u}}=\overline{T}$, we would then have
$\lambda(\overline{T})=1$ and $\dot \lambda>0$ for $t\geq \overline{T}$, so
that $u$ directly goes to $0$: no actual switch to $u=\bar{u}$ has taken
place. Using the constancy of the Hamiltonian, we can conclude that a direct
switch from $u_{\sigma}$ to $u=0$ is not possible at $t_{\sigma
  0}\leq\overline{T}$. Indeed, if such an early switch occurred, we would have

\[
\lambda_0\left(f(x_0)-rx_0\right)=\lambda_{\overline{T}}\left(f(x_{\overline{T}})-r x_{\overline{T}} \right)\;,
\]
and
\[
\lambda_0\left(-rx_0\right)=\lambda_{\overline{T}}\left(-r x_{\overline{T}} \right)\;.
\]
Taking the differences between the equalities, we obtain
\[
\lambda_0f(x_0)=\lambda_{\overline{T}}f(x_{\overline{T}})\;.
\]
Taking the quotient of the last two equalities then yields
\[
\frac{f(x_0)}{x_0}=\frac{f(x_{\overline{T}})}{x_{\overline{T}}}\;,
\]
which we have shown earlier be impossible. The optimal solution then leaves the singular arc with $u=\bar{u}$ strictly before $\overline{T}$, and switches to $0$ strictly after $\overline{T}$.

From then on, things are unchanged with respect to the bang-bang case. 

{\bf Solution with $\lambda_0\leq 1$: }

Since $\dot\lambda>0$ in $\lambda=1$ in the dark phase, such a solution would
mean that harvesting takes place during the whole dark phase because no
transition from $u=0$ to $u=\bar{u}$ can take place in this phase
(Proposition~\ref{prop1}-(i)). Two possibilities then occur: either
$u=\bar{u}$ all the time or switches from $u=\bar{u}$ to $u=0$ or $u_\sigma$
and then back to $u=\bar{u}$ take place in the interval $(0,\overline{T})$. 

In the latter case, the first switch from  $u=\bar{u}$ to $u=0$ can only take
place with $x>x_\sigma$ (Proposition~\ref{prop1}-(iii)). Then, when the
control $u=0$ is applied for some time, the solution $x(t)$ is increasing. No
switch back to $u=\bar{u}$ can then take place before $\overline{T}$ because
such a switch would require $x(t)<x_\sigma$ (Proposition \ref{prop1}-(ii)),
which cannot occur. An optimal solution of this form cannot exist.  

We can also show that no strategy in the $(0,\overline{T})$ interval can have
the form $u=\bar{u}\rightarrow u_\sigma \rightarrow u=0 \text{ or }
\bar{u}$. Indeed, in order to reach the singular arc with $u=\bar{u}$, a
solution should be coming from above it. If the switch that takes place at the
end of the singular phase is from $u_\sigma$ to $0$, $x(t)$ will increase and
there should be a subsequent switch from $0$ to $\bar{u}$ which is impossible
with $x(t)>x_\sigma$. If the switch that takes place 
at the end of the singular phase is from $u_\sigma$ to $\bar{u}$, $x(t)$ will
decrease all the time between $t_{\sigma \bar{u}}$ and $T$, which is in
contradiction with the fact that we had $x(0)>x_\sigma$.


The only potential optimal control in that family is therefore $u(t)=\bar{u}$
for all times. Using the expressions computed previously, this control can be
a candidate optimal control law only if $x_0=x_{0min}$ as we have seen
earlier. This solution can only potentially exist if (\ref{x0max}) is
satisfied.

\end{document}